\documentclass[a4paper,twoside]{article}
\usepackage{a4}
\usepackage{amssymb}
\usepackage{amsmath}
\usepackage{url}
\usepackage{upref}
\usepackage[active]{srcltx}
\usepackage[dvipsnames]{xcolor} 
\usepackage[pagebackref,colorlinks,citecolor=blue,linkcolor=blue,urlcolor=blue]{hyperref}
\allowdisplaybreaks[2] 
\newcount\minutes \newcount\hours
\hours=\time
\divide\hours 60
\minutes=\hours
\multiply\minutes -60
\advance\minutes \time
\newcommand{\klockan}{\the\hours:{\ifnum\minutes<10 0\fi}\the\minutes}
\newcommand{\tid}{\today\ \klockan}
\newcommand{\prtid}{\smash{\raise 10mm \hbox{\LaTeX ed \tid}}}
\renewcommand{\prtid}{}
\makeatletter
\def\sectionmark#1{} 
\def\subsectionmark#1{}
\newcommand{\sectnr}{\ifnum \c@secnumdepth >\z@
                 \thesection.\hskip 1em\relax \fi}
\def\@evenhead{\footnotesize\rm\thepage\hfil\leftmark\hfil\llap{\prtid}}
\def\@oddhead{\footnotesize\rm\rlap{\prtid}\hfil\rightmark\hfil\thepage}
\def\tableofcontents{\section*{Contents} 
 \@starttoc{toc}}
\makeatother
\makeatletter
\def\@biblabel#1{#1.}
\makeatother
\makeatletter
\let\Thebibliography=\thebibliography
\renewcommand{\thebibliography}[1]{\def\@mkboth##1##2{}\Thebibliography{#1}
\addcontentsline{toc}{section}{References}
\frenchspacing 
\setlength{\@topsep}{0pt}
\setlength{\itemsep}{0pt}%
\setlength{\parskip}{0pt plus 2pt}%
}
\makeatother
\makeatletter
\def\mdots@{\mathinner.\nonscript\!.%
 \ifx\next,.\else\ifx\next;.\else\ifx\next..\else
 \nonscript\!\mathinner.\fi\fi\fi}
\let\ldots\mdots@
\let\cdots\mdots@
\let\dotso\mdots@
\let\dotsb\mdots@
\let\dotsm\mdots@
\let\dotsc\mdots@
\def\vdots{\vbox{\baselineskip2.8\p@ \lineskiplimit\z@
    \kern6\p@\hbox{.}\hbox{.}\hbox{.}\kern3\p@}}
\def\ddots{\mathinner{\mkern1mu\raise8.6\p@\vbox{\kern7\p@\hbox{.}}%
    \raise5.8\p@\hbox{.}\raise3\p@\hbox{.}\mkern1mu}}
\makeatother
\makeatletter
\let\Enumerate=\enumerate
\renewcommand{\enumerate}{\Enumerate%
\setlength{\@topsep}{0pt}
\setlength{\itemsep}{0pt}%
\setlength{\parskip}{0pt plus 1pt}%
\renewcommand{\theenumi}{\textup{(\alph{enumi})}}%
\renewcommand{\labelenumi}{\theenumi}%
}
\let\endEnumerate=\endenumerate
\renewcommand{\endenumerate}{\endEnumerate\unskip}
\makeatother
\newcommand{\addjustenumeratemargin}[1]{%
\setbox0\hbox{(a)} 
\setbox1\hbox{#1} 
\addtolength{\leftmargini}{-\wd0}%
\addtolength{\leftmargini}{\wd1}%
}
\makeatletter
\def\@seccntformat#1{\csname the#1\endcsname.\quad}
\makeatother
\RequirePackage{ifthen}
\newcommand{\authortitle}[3]{\author{#1}\title{#2}%
   \ifthenelse{\equal{#3}{}}{\markboth{#1}{#2}}{\markboth{#1}{#3}}}
\newcommand{\auth}[2]{{#1, #2.}}
\newcommand{\art}[6]{{\sc #1, \rm #2, \it #3\/ \bf #4 \rm (#5), \mbox{#6}.}}
\newcommand{\artin}[3]{{\sc #1, \rm #2,  in #3.}}
\newcommand{\artprep}[3]{{\sc #1, \rm #2, \rm #3.}}

\newcommand{\book}[3]{{\sc #1, \it #2, \rm #3.}}
\newcommand{\AND}{{\rm and }}
\newcommand{\arXiv}[1]{{\tt \href{https://arxiv.org/abs/#1}{arXiv:#1}}}

\RequirePackage{amsthm}
\newtheoremstyle{descriptive}%
  {\topsep}   
  {\topsep}   
  {\rmfamily} 
  {}          
  {\bfseries} 
  {.}         
  { }         
  {}          
\newtheoremstyle{propositional}%
  {\topsep}   
  {\topsep}   
  {\itshape}  
  {}          
  {\bfseries} 
  {.}         
  { }         
  {}          
\theoremstyle{propositional}
\newtheorem{thm}{Theorem}[section]
\newtheorem{prop}[thm]{Proposition}
\newtheorem{lem}[thm]{Lemma}
\newtheorem{cor}[thm]{Corollary}
\newtheorem{theorem}[thm]{Theorem}
\theoremstyle{descriptive}
\newtheorem{deff}[thm]{Definition}
\newtheorem{example}[thm]{Example}
\newtheorem{remark}[thm]{Remark}
\makeatletter
\renewenvironment{proof}[1][\proofname]{\par
  \pushQED{\qed}%
  \normalfont
  \trivlist
  \item[\hskip\labelsep
        \itshape
    #1\@addpunct{.}]\ignorespaces
}{%
  \popQED\endtrivlist\@endpefalse
}
\makeatother
\newcommand{\setm}{\setminus}
\renewcommand{\emptyset}{\varnothing}
\def\vint{\mathop{\mathchoice%
          {\setbox0\hbox{$\displaystyle\intop$}\kern 0.22\wd0%
           \vcenter{\hrule width 0.6\wd0}\kern -0.82\wd0}%
          {\setbox0\hbox{$\textstyle\intop$}\kern 0.2\wd0%
           \vcenter{\hrule width 0.6\wd0}\kern -0.8\wd0}%
          {\setbox0\hbox{$\scriptstyle\intop$}\kern 0.2\wd0%
           \vcenter{\hrule width 0.6\wd0}\kern -0.8\wd0}%
          {\setbox0\hbox{$\scriptscriptstyle\intop$}\kern 0.2\wd0%
           \vcenter{\hrule width 0.6\wd0}\kern -0.8\wd0}}%
          \mathopen{}\int}
\DeclareMathOperator{\capp}{cap}
\newcommand{\cp}{\capp_p}
\newcommand{\binfty}{{\boldsymbol{\infty}}}
\newcommand{\grad}{\nabla}
\newcommand{\lQ}{\itunderline{Q}}
\newcommand{\uQ}{\itoverline{Q}}
\newcommand{\lQo}{\itunderline{Q}_0}
\newcommand{\uQo}{\itoverline{Q}_0}
\newcommand{\wt}{\widetilde{w}}

\DeclareMathOperator{\Lip}{Lip}
\newcommand{\Lipc}{{\Lip_c}}
\newcommand{\Liploc}{{\Lip\loc}}
\DeclareMathOperator{\diam}{diam}
\DeclareMathOperator{\dvg}{div}
\newcommand{\bdry}{\partial}
\newcommand{\bdy}{\bdry}
\newcommand{\loc}{_{\rm loc}}

{\catcode`p =12 \catcode`t =12 \gdef\eeaa#1pt{#1}}      
\def\accentadjtext#1{\setbox0\hbox{$#1$}\kern   
                \expandafter\eeaa\the\fontdimen1\textfont1 \ht0 }
\def\accentadjscript#1{\setbox0\hbox{$#1$}\kern 
                \expandafter\eeaa\the\fontdimen1\scriptfont1 \ht0 }
\def\accentadjscriptscript#1{\setbox0\hbox{$#1$}\kern   
                \expandafter\eeaa\the\fontdimen1\scriptscriptfont1 \ht0 }
\def\accentadjtextback#1{\setbox0\hbox{$#1$}\kern       
                -\expandafter\eeaa\the\fontdimen1\textfont1 \ht0 }
\def\accentadjscriptback#1{\setbox0\hbox{$#1$}\kern     
                -\expandafter\eeaa\the\fontdimen1\scriptfont1 \ht0 }
\def\accentadjscriptscriptback#1{\setbox0\hbox{$#1$}\kern 
                -\expandafter\eeaa\the\fontdimen1\scriptscriptfont1 \ht0 }
\def\itoverline#1{{\mathsurround0pt\mathchoice
        {\rlap{$\accentadjtext{\displaystyle #1}
                \accentadjtext{\vrule height1.593pt}
                \overline{\phantom{\displaystyle #1}
                \accentadjtextback{\displaystyle #1}}$}{#1}}
        {\rlap{$\accentadjtext{\textstyle #1}
                \accentadjtext{\vrule height1.593pt}
                \overline{\phantom{\textstyle #1}
                \accentadjtextback{\textstyle #1}}$}{#1}}
        {\rlap{$\accentadjscript{\scriptstyle #1}
                \accentadjscript{\vrule height1.593pt}
                \overline{\phantom{\scriptstyle #1}
                \accentadjscriptback{\scriptstyle #1}}$}{#1}}
        {\rlap{$\accentadjscriptscript{\scriptscriptstyle #1}
                \accentadjscriptscript{\vrule height1.593pt}
                \overline{\phantom{\scriptscriptstyle #1}
                \accentadjscriptscriptback{\scriptscriptstyle #1}}$}{#1}}}}
\def\itunderline#1{{\mathsurround0pt\mathchoice
        {\rlap{$\underline{\phantom{\displaystyle #1}
                \accentadjtextback{\displaystyle #1}}$}{#1}}
        {\rlap{$\underline{\phantom{\textstyle #1}
                \accentadjtextback{\textstyle #1}}$}{#1}}
        {\rlap{$\underline{\phantom{\scriptstyle #1}
                \accentadjscriptback{\scriptstyle #1}}$}{#1}}
        {\rlap{$\underline{\phantom{\scriptscriptstyle #1}
                \accentadjscriptscriptback{\scriptscriptstyle #1}}$}{#1}}}}
\newcommand{\al}{\alpha}
\newcommand{\alp}{\alpha}
\newcommand{\be}{\beta}
\newcommand{\clB}{\itoverline{B}}

\newcommand{\ga}{\gamma}
\newcommand{\de}{\delta}
\newcommand{\eps}{\varepsilon}
\newcommand{\la}{\lambda}
\newcommand{\Om}{\Omega}
\newcommand{\om}{\omega}
\renewcommand{\phi}{\varphi}
\newcommand{\p}{{$p\mspace{1mu}$}}
\newcommand{\R}{\mathbf{R}}
\newcommand{\Rn}{\R^n}
\newcommand{\Sp}{\mathbf{S}}
\newcommand{\Spn}{\mathbf{S}^{n-1}}
\newcommand{\Spe}{\mathbf{S}^{1}}
\newcommand{\eR}{{\overline{\R}}}
\newcommand{\Np}{N^{1,p}}
\newcommand{\Nploc}{N^{1,p}\loc}
\newcommand{\Lploc}{L^p\loc}
\newcommand{\Ga}{\Gamma}
\newcommand{\ut}{{\tilde{u}}}
\newcommand{\ub}{{\bar{u}}}
\newcommand{\clG}{\itoverline{G}}
\newcommand{\CPI}{C_{\rm PI}}
\newcommand{\Cp}{{C_p}}
\newcommand{\Cpw}{{C_{p,w}}}

\newcommand{\simge}{\gtrsim}
\newcommand{\simle}{\lesssim}

\newcommand{\chione}{\chi}
\newcommand{\Xstar}{X^*}
\newcommand{\bdystar}{\partial^*}
\newcommand{\UU}{\mathcal{U}}
\newcommand{\uP}{\itoverline{P}}     
\newcommand{\lP}{\itunderline{P}} 
\newcommand{\gadot}{\dot{\ga}}

\numberwithin{equation}{section}
\newenvironment{ack}{\medskip{\it Acknowledgement.}}{}

\begin{document}

\authortitle{Anders Bj\"orn, Jana Bj\"orn, Sylvester Eriksson-Bique and Xiaodan Zhou}
{Uniqueness and nonuniqueness of \p-harmonic Green functions on 
weighted $\R^n$ and metric spaces}
{Uniqueness and nonuniqueness of \p-harmonic Green functions}

 \author{
 Anders Bj\"orn \\
 \it\small Department of Mathematics, Link\"oping University, SE-581 83 Link\"oping, Sweden\\
 \it\small and Theoretical Sciences Visiting Program, Okinawa Institute of\\
 \it\small
  Science and Technology Graduate University, Onna, 904-0495, Japan \\
 \it \small anders.bjorn@liu.se, ORCID\/\textup{:} 0000-0002-9677-8321
 \\  \\
 Jana Bj\"orn \\
 \it\small Department of Mathematics, Link\"oping University, SE-581 83 Link\"oping, Sweden\\
 \it\small and Theoretical Sciences Visiting Program, Okinawa Institute of\\
 \it\small
  Science and Technology Graduate University, Onna, 904-0495, Japan \\
 \it \small jana.bjorn@liu.se, ORCID\/\textup{:} 0000-0002-1238-6751
 \\  \\
Sylvester Eriksson-Bique \\
\it\small Department of Mathematics and Statistics, University of Jyv\"askyl\"a,\\
\it\small P.O. Box 35 (MaD), FI-40014 University of Jyv\"askyl\"a, Finland \\
 \it\small and Theoretical Sciences Visiting Program, Okinawa Institute of\\
 \it\small
  Science and Technology Graduate University, Onna, 904-0495, Japan \\
\it \small  sylvester.d.eriksson-bique@jyu.fi  ORCID\/\textup{:} 0000-0002-1919-6475
\\ \\
Xiaodan Zhou \\
\it \small Analysis on Metric Spaces Unit\\
\it \small
Okinawa Institute of
  Science and Technology Graduate University,  \\
\it \small Onna, 904-0495, Japan \\
\it \small xiaodan.zhou@oist.jp  ORCID\/\textup{:} 0000-0003-2169-7524
 }


\date{Preliminary version, \today}
\date{} 

\maketitle

\noindent{\small {\bf Abstract.} 
We study uniqueness
of \p-harmonic Green functions in 
domains $\Om$ in 
a complete metric space equipped with a doubling measure
supporting a \p-Poincar\'e inequality, 
with $1<p<\infty$.
For bounded domains in unweighted $\Rn$,
the uniqueness was shown for the \p-Laplace operator $\Delta_p$ and all $p$ by 
Kichenassamy--V\'eron (\emph{Math.\ Ann.} {\bf 275} (1986), 599--615),
while for $p=2$ it is an easy consequence of the linearity of the Laplace
operator~$\Delta$.
Beyond that, uniqueness is only known in some particular cases, such as in
Ahlfors \p-regular spaces, as shown by 
Bonk--Capogna--Zhou ({\tt arXiv:2211.11974}).
When the singularity $x_0$ has positive \p-capacity,
the Green function is a particular multiple 
of the capacitary potential for $\cp(\{x_0\},\Om)$ and is therefore unique.
Here we give a sufficient condition for uniqueness in metric spaces, and provide
an example showing that the range of $p$ for which it holds
(while $x_0$ has zero \p-capacity) can be a nondegenerate interval.
In the opposite direction, we give the first example
showing that uniqueness can fail in metric spaces, even for $p=2$.
}

\medskip
\noindent
{\small \emph{Key words and phrases}:
doubling measure,
metric space,
nonuniqueness,
\p-harmonic Green function,
Poincar\'e inequality,
weighted Euclidean space,
uniqueness.
}




\section{Introduction}

In the Euclidean space $\R^n$, a function $u$ is said to be 
the (\p-harmonic) Green function for a bounded domain $\Omega\subset \R^n$ with singularity at $x_0\in\Om$
if it is the weak solution of the equation 
\begin{equation}\label{p-Lap}
-  \Delta_p u  = \de_{x_0}  \quad \text{in $\Omega$}
\end{equation}
with the Dirac measure $\de_{x_0}$  in the right-hand side and zero boundary values on $\bdy\Om$
in a weak sense.
Here
\[
\Delta_p u  := \text{div}(|\nabla u|^{p-2} \nabla u),  \quad 1<p<\infty,
\]
is the \p-Laplace operator. 
The Green function is \p-harmonic outside of its singularity,
that is, it satisfies the \p-Laplace equation $\Delta_p u  =0$.

In metric spaces, Holopainen--Shan\-mu\-ga\-lin\-gam~\cite{HS02} introduced 
a notion of singular functions which extended the notion of Green functions from 
$\R^n$ and Riemannian manifolds to more general settings. 
More precise definitions and various characterizations of Green functions in metric spaces 
were later given in Bj\"orn--Bj\"orn--Lehrb\"ack~\cite{BBLgreen}
and Bj\"orn--Bj\"orn~\cite{BBglobal}, 
see Definitions~\ref{deff-sing-bdd} and~\ref{deff-sing-unbdd} below.
For  domains in complete metric spaces $X=(X,d)$, 
equipped with a doubling measure $\mu$ supporting a \p-Poincar\'e inequality
with $1<p< \infty$, 
it was shown  in~\cite{BBglobal}  and~\cite{BBLgreen} that 
Green functions exist in a domain $\Om$
if and only if the \p-capacity $\Cp(X \setm \Om)>0$
or $X$ is \p-hyperbolic.

Roughly speaking, by \cite[Theorem~8.5]{BBLgreen}, $u:\Om\to(0,\infty]$ is
a (\p-harmonic) \emph{Green function}   with singularity at $x_0\in\Om$ for a (bounded) domain $\Om$ 
in a metric space $X$ if it is \p-harmonic in 
$\Om\setm\{x_0\}$, has zero boundary values on $\bdy \Om$
in a suitable sense and is properly normalized 
so that
 \begin{equation*}
  \cp(\{x: u(x) \ge b\},\Om) = b^{1-p},
 \quad \text{when }
    0  <b < u(x_0)=\lim_{x\to x_0} u(x).
 \end{equation*}
This normalization captures in a quantitative way the right-hand side $\de_{x_0}$ in~\eqref{p-Lap}.

In this paper we study whether such Green functions in metric spaces are unique.
Our first result answers this question in the affirmative
in spaces satisfying a critical volume growth condition. 
Here $B_r=B(x_0,r)=\{y \in X : d(y,x_0)<r\}$.

\begin{thm}   \label{thm-unique-Green-intro}
Let $X$ be a complete metric space and assume that $\mu$ is doubling 
and supports a \p-Poincar\'e inequality on $X$, where $1<p<\infty$.
Assume that 
\begin{equation}  \label{eq-ass-for-unique-limsup}
\limsup_{r \to 0} \biggl( \frac{\mu(B_{r})}{r^p} \biggr)^{1/(p-1)}  
\int_{r}^1 \biggl( \frac{\rho^p}{\mu(B_\rho)} \biggr)^{1/(p-1)} \frac{d\rho}{\rho} 
= \infty.
\end{equation}

If $u$ and $v$ are\/ \textup(\p-harmonic\/\textup) Green functions in a  domain $\Om$ with singularity at~$x_0\in\Om$,
then $u=v$, i.e.\ the Green function in $\Om$ with singularity at $x_0$ is unique.
\end{thm}

\begin{remark}\label{mainthm-remark-intro}
When $\Cp(\{x_0\})>0$, the Green function is a particular multiple 
of the capacitary potential for $\cp(\{x_0\},\Om)$ and is therefore unique, see
Bj\"orn--Bj\"orn--Lehrb\"ack~\cite[Theorem~1.3]{BBLgreen}
(for bounded $\Om$) and
Bj\"orn--Bj\"orn~\cite[Corollary~10.3]{BBglobal}
(for general $\Om$).   

Theorem~\ref{thm-unique-Green-intro}
provides a new sufficient condition \eqref{eq-ass-for-unique-limsup} for the uniqueness of the Green function in the case $\Cp(\{x_0\})=0$. When the space $X$ is Ahlfors $Q$-regular and $p=Q$,  \eqref{eq-ass-for-unique-limsup} is satisfied and our result recovers the uniqueness of Green functions on a bounded regular
domain $\Omega\subset X$ in Bonk--Capogna--Zhou~\cite[Theorem~1.1(ii) and Remark~1.1]{BCZ}.

More generally, condition \eqref{eq-ass-for-unique-limsup} holds if $p>q_0$, where
\[
q_0 := \sup \biggl\{   q>0: \frac{\mu(B_{r})}{\mu(B_{R})} \simle 
\Bigl( \frac{r}{R} \Bigr)^q  \text{ for all }
0<r<R\le1 \biggr\},
\] 
see 
Proposition~\ref{prop-uQ-lQ}\ref{k-a}.
Example~\ref{ex-abcd-alt} shows that
the range of new exponents $p$
covered by Theorem~\ref{thm-unique-Green-intro}
can be a nondegenerate arbitrarily large interval.

\end{remark}

At the same time, in Example~\ref{ex-Finsler-1} we  prove the following nonuniqueness result
with a \p-admissible weight in~$\R^2$.

\begin{theorem}      \label{thm-nonunique}
Let $1<p<\infty$ and either $\Om=\{x\in \R^2: |x|<R\}$   
or $\Om=\R^2$.

Then there is a metric $d$ on $\R^2$,  biLipschitz equivalent to the Euclidean metric,
and a doubling 
measure $d\mu= w\,dx$ supporting a \p-Poincar\'e inequality,
such that for the metric space $X:=(\R^2,d,\mu)$
there are uncountably many\/ \textup(\p-harmonic\/\textup)
Green functions in $\Om$ with a singularity  at $0$,
that is, the Green functions in $\Om$ are not unique.

The metric $d$ is independent of $p$ and $\Om$.
If $1<p<2$ then we can choose $w\equiv 1$.
\end{theorem}

For domains in $\R^n$, equipped with the Euclidean metric and the Lebesgue measure, 
Green functions are known to be unique.
For $p=2$, this is an easy consequence of the linearity of the Laplace operator and the maximum 
principle for harmonic functions.
When $p\ne 2$, the \p-Laplace operator is nonlinear and the above argument does not apply.
In this case, uniqueness was proved only in 1986 by 
Kichenassamy--V\'eron~\cite[Theorem~2.1]{KichVeron}. 

Beyond that, uniqueness has been shown in the following cases:
 for regular relatively compact domains in $n$-dimensional
Riemannian manifolds (and $p=n$) by Holopainen~\cite[Theorem~3.22]{Ho},
and more
recently for bounded regular domains in Ahlfors $Q$-regular metric spaces (and  $p=Q$)
in Bonk--Capogna--Zhou~\cite[Theorem~1.1]{BCZ}.
The uniqueness also holds when $\Cp(\{x_0\})>0$, by 
\cite[Corollary~10.3]{BBglobal} and \cite[Theorem~1.3]{BBLgreen},
since the Green function is then
just a suitable multiple of the capacitary potential for $\{x_0\}$ in $\Om$.

The \emph{\p-harmonic functions}
in a general metric measure space $(X,d,\mu)$ are defined as 
continuous local minimizers   of the \p-energy integral 
\begin{equation}  \label{eq-p-energy}
\int_\Om g_u^p\,d\mu,
\end{equation}
where $g_u$ is the minimal \p-weak upper gradient of $u$,
see Sections~\ref{sect-ug} and~\ref{sect-pharm}.
This is a natural generalization of the fact that the \p-Laplace equation $\Delta_p u=0$ in $\R^n$ is the Euler--Lagrange
equation for minimizers of the \p-energy $\int_\Om |\grad u|^p\,dx$.

If $X$ is equipped with a smooth differentiable structure, such as for
Riemannian manifolds, then one can derive 
an Euler--Lagrange equation which characterizes \p-harmonic functions.
In particular, if $\R^n$ is equipped with a \p-admissible weight~$w$, 
as in Heinonen--Kilpel\"ainen--Martio~\cite{HeKiMa},
then \p-harmonic functions are local minimizers of the weighted \p-energy integral
\[
\int_\Om |\grad u(x)|^p w(x) \,dx
\]
and solve the weighted \p-Laplace equation 
\[
\dvg (w(x) |\grad u|^{p-2} \grad u ) =0.  
\]
In general metric measure spaces, such a description is not always possible -- 
see however Gigli--Mondino~\cite{GigMon}. 

A seminal result of Cheeger~\cite{Che99} states that every 
complete metric measure space $(X,d,\mu)$,
with $\mu$ doubling and supporting a \p-Poincar\'e inequality, can
be equipped with a  \emph{measurable differentiable structure}. 
In such cases, the \p-energy can be modified to yield an Euler--Lagrange-type equation. 
This is obtained via differential calculus, as developed in \cite{Che99}.
Namely, one can define a vector-valued differential $Du$, 
as a measurable section of a measurable tangent bundle, 
whose pointwise Euclidean norm $|Du|$ is comparable to the minimal \p-weak upper gradient  $g_u$ of  $u$.
For such a choice of a fixed measurable differentiable structure, 
functions minimizing the Cheeger \p-energy 
\[
\int_\Om |Du|^p \,d\mu
\]
satisfy the equation
\begin{equation}   \label{pde-eq}
\int_\Om |Du|^{p-2}Du\cdot D\phi =0  \quad \text{for all  functions $\phi\in \Lip_0(\Om)$}
\end{equation}
and their continuous representatives  are called \emph{Cheeger \p-harmonic functions}. 

If $X$ is infinitesimally Hilbertian, as in e.g.~\cite{GigMon}, the inner product in~\eqref{pde-eq} 
can be chosen in such a way that 
\[
|Du|^2 := Du \cdot Du =g_u^2 
\] 
at almost every point,  see also \cite[Theorem~6.1 and Corollary~4.41]{Che99}. 
In this case, Cheeger \p-harmonic and (upper gradient) \p-harmonic functions 
coincide, see  \cite[Theorem~4.2]{GigMon}. 
This case occurs for example in all analytically 
one-dimensional cases, such as the Laakso spaces 
considered in~\cite{La00}, see also 
Cheeger--Kleiner~\cite[Theorem~9.1]{ChKl}.

On the other hand, in Theorem~\ref{thm-nonunique}
we do not have a scalar product, and as a result, 
the (upper gradient) \p-harmonic Green functions    
differ significantly  from the Cheeger \p-harmonic Green functions. 
From the metric perspective, the space $(\R^2,d)$ in Theorem~\ref{thm-nonunique} 
is biLipschitz equivalent to the Euclidean space, 
but from a PDE point of view their Laplacians are very different. 
The metric $d$ 
is defined through an $\ell^1$ norm in the tangent bundle
and is similar to a Finsler structure studied in Finsler geometry 
-- see e.g.\  
\cite{Bao}, \cite{Chern} and \cite{Rademacher} for introductions to this vast field.

The notion of  \p-harmonic functions 
(based on \eqref{eq-p-energy})
only depends on  the metric 
and the measure of the underlying space. 
In contrast, the definition of Cheeger \p-harmonic functions also depends on the choice 
of a measurable differentiable structure. 
In particular, when the space fails to be infinitesimally Hilbertian, 
the upper gradient $g_u$ in \eqref{eq-p-energy}
is only comparable to the $|Du|$ norm and there is no Euler--Lagrange equation for 
the \p-harmonic functions
(based on \eqref{eq-p-energy}), which therefore differ from the 
Cheeger \p-harmonic functions.

Note, however, that any ``positive''
result proved for the energy minimizing \p-harmonic functions
(such as our uniqueness Theorem~\ref{thm-unique-Green-intro})
holds also for the Cheeger \p-harmonic functions. 
Indeed, it suffices to replace $g_u$ by $|Du|$ in all the proofs.
In particular, Theorem~\ref{thm-unique-Green-intro} provides new uniqueness
results for Green functions in weighted $\R^n$ as in Heinonen--Kilpel\"ainen--Martio~\cite{HeKiMa}, 
cf.\ Example~\ref{ex-abcd-alt}.
However, we do not know whether  nonuniqueness as in Theorem~\ref{thm-nonunique}
can be obtained for Cheeger \p-harmonic Green functions.

The examples in Theorem~\ref{thm-nonunique}
also yield counterexamples to other problems. 
The strong comparison principle for the  \p-Laplace operator $\Delta_p$,
$1<p<\infty$, states that 
if $u$ and $v$ are two \p-harmonic functions in a domain $\Om$ satisfying $u\ge v$, 
then either $u> v$ in $\Om$ or $u \equiv v$ in $\Om$. 
When $p=2$, this is a direct consequence of the linearity of the Laplace operator
and the maximum principle for harmonic functions. 
For the nonlinear \p-Laplace operator when $p\ne2$, this argument does not hold.
Using the theory of quasiregular mappings, 
Manfredi~\cite[Theorem~2]{Man88} proved the strong comparison principle   
for planar domains (in unweighted $\R^2$). 
In higher dimensional Euclidean spaces $\R^n$, with $n\ge 3$, 
the strong comparison principle is still an open problem. 
In metric spaces, we have the following negative result.

\begin{cor} \label{cor-nonls}
Let $1<p<\infty$.
Then there is a metric $d$ on  $\R^2$,  biLipschitz equivalent to the Euclidean metric,
and a \p-admissible measure $d\mu= w\,dx$,  
such that the following hold for $X=(\R^2,d,\mu)$:
\begin{enumerate}
\item The strong comparison principle 
fails for \p-harmonic functions in the punctured unit disc $\{x\in\R^2: 0<|x|<1\}$.
\item The theory for $2$-harmonic functions is nonlinear\/\textup:
there are $2$-harmonic functions $u$ and $v$ 
in the punctured unit disc $\{x\in\R^2: 0<|x|<1\}$ such 
that $u+v$ is not $2$-harmonic.
\end{enumerate}
\end{cor}

The paper is organized in the following way. 
Section~\ref{sect-ug} contains some definitions and preliminary results 
about upper gradients and Newtonian functions.
In Section~\ref{sect-pharm}, \p-harmonic functions and some auxiliary results for them are presented. 
Section~\ref{sect-sing-Green} deals with the definitions of
singular and Green functions and their properties.
The uniqueness Theorem~\ref{thm-unique-Green-intro} is proved in Section~\ref{sect-unique},
while Theorem~\ref{thm-nonunique} and Corollary~\ref{cor-nonls} are proved in 
Section~\ref{sect-nonuniq}.

\begin{ack} 
Most of this
research was conducted while the first three authors were visiting 
the Okinawa Institute of Science and Technology (OIST) through 
the Theoretical Sciences Visiting Program (TSVP) in 2024. 
We thank OIST and TSVP  for their hospitality and support.
A.~B. and J.~B. were also supported by the Swedish Research Council,
  grants 2020-04011 and 2022-04048, respectively. 
  S.~\mbox{E.-B.}
was supported in part by the Research Council of Finland grant 354241. 
X.~Z. was supported by JSPS Grant-in-Aid for Early-Career Scientists 
No.~22K13947 and JSPS Grant-in-Aid for Scientific Research(C) No.~25K00211.
\end{ack}

\section{Upper gradients and Newtonian spaces}
\label{sect-ug}

\emph{We assume throughout the paper that $1 < p<\infty$ and that
$X=(X,d,\mu)$ is a metric space equipped
with a metric $d$ and a positive complete  Borel  measure $\mu$
such that  $0<\mu(B)<\infty$ for all  balls $B \subset X$.
We also assume that $\diam X >0$. Note that it follows from these  assumptions that $X$ is separable.
Additional standing
assumptions are added at  the beginning of Section~\ref{sect-pharm}.}

\medskip

In this section, we introduce
the necessary metric space concepts used in this paper.
Proofs of most of the results mentioned in 
this section can be found in the monographs
Bj\"orn--Bj\"orn~\cite{BBbook} and
Heinonen--Koskela--Shan\-mu\-ga\-lin\-gam--Tyson~\cite{HKSTbook}.

We denote by $B(x,r):=\{y \in X : d(x,y)<r\}$
the open ball of radius $r>0$ centred at $x\in X$. 
For $B:=B(x,r)$ we use the notation $\la B=B(x, \la r)$. 
In some parts, we will also write $B_r$, when the centre of the ball is fixed.
In metric spaces
it can happen that balls with different centres or
radii denote the same set.
We will, however,
make the convention that a ball $B$ comes with a predetermined
centre and radius.
Unless said otherwise, balls are assumed to be open in this paper.
By $\clB$ we mean the closure of the open ball $B=B(x,r)$,
not the (possibly larger) set $\{y:d(x,y)\le r\}$.
We write $\diam A$ for the diameter of a set $A\subset X$. 

The measure $\mu$ is \emph{doubling} 
if there exists a \emph{doubling constant} $C_\mu\ge 1$ such that 
\[
  \mu(2B)\le C_\mu\mu(B) 
\quad \text{for every open ball }B.
\]
The measure   $\mu$ is \emph{Ahlfors $Q$-regular} 
if there is a  constant $C\ge 1$ such that
\[
C^{-1}r^Q\le\mu(B(x,r))\le Cr^Q
\quad
\text{for all $x\in X$ and $0<r< 2 \diam X$}.
\]

A \emph{curve} is a continuous mapping from an interval,
and a \emph{rectifiable} curve is a curve with finite length.
A rectifiable curve can be parameterized by its arc length $ds$. 
A property holds for \emph{\p-almost every curve}
if the curve family $\Ga$ for which it fails has zero \p-modulus,
i.e.\ there is $\rho\in L^p(X)$ such that
$\int_\ga \rho\,ds=\infty$ for every $\ga\in\Ga$.

\begin{deff} \label{deff-ug}
Let  $u: X \to \eR:=[-\infty,\infty]$ be a function.
A Borel function $g: X \to [0,\infty]$
is an \emph{upper gradient}  of $u$
if for all nonconstant rectifiable curves 
$\gamma : [0,l_{\gamma}] \to X$,
\begin{equation} \label{ug-cond}
        |u(\gamma(0)) - u(\gamma(l_{\gamma}))| \le \int_{\gamma} g\,ds,
\end{equation}
where the left-hand side is considered to be $\infty$ 
whenever at least one of the 
terms therein is infinite.
A measurable function $g: X \to [0,\infty]$ 
is a \emph{\p-weak upper gradient} of~$u$
if it satisfies \eqref{ug-cond} for \p-almost every 
nonconstant rectifiable
curve $\ga$.
\end{deff}

The upper gradients were introduced in 
Heinonen--Koskela~\cite{HeKo96}, \cite{HeKo98} 
and  \p-weak upper gradients in  Koskela--MacManus~\cite{KoMc}. 
 It was also shown in \cite{KoMc} 
 that if $g \in \Lploc(X)$ is a \p-weak upper gradient of $u$,
 then one can find a sequence $\{g_j\}_{j=1}^\infty$
 of upper gradients of $u$ such that $\|g_j-g\|_{L^p(X)} \to 0$.
 If $u$ has an upper gradient in $\Lploc(X)$, then
 it has an a.e.\ unique \emph{minimal \p-weak upper gradient} $g_u \in \Lploc(X)$
 in the sense that $g_u \le g$ a.e.\ for every
 \p-weak upper gradient $g \in \Lploc(X)$ of $u$,
 see Shan\-mu\-ga\-lin\-gam~\cite{Sh-harm}.

Following Shan\-mu\-ga\-lin\-gam~\cite{Sh-rev}, 
 we define a version of Sobolev spaces on the metric space $X$.

\begin{deff} \label{deff-Np}
For a measurable function $u: X\to \eR$,
let 
\[
        \|u\|_{\Np(X)} = \biggl( \int_X |u|^p \, d\mu 
                + \inf_g  \int_X g^p \, d\mu \biggr)^{1/p},
\]
where the infimum is taken over all upper gradients $g$ of $u$.
The \emph{Newtonian space} on $X$ is 
\[
        \Np (X) = \{u :  \|u\|_{\Np(X)} <\infty \}.
\]
\end{deff}

The quotient
space $\Np(X)/{\sim}$, where  $u \sim v$ if and only if $\|u-v\|_{\Np(X)}=0$,
is a Banach space and a lattice, see Shan\-mu\-ga\-lin\-gam~\cite{Sh-rev}.

In this paper it is convenient to assume that functions in $\Np(X)$
 are defined everywhere (with values in $\eR$),
not just up to an equivalence class in the corresponding quotient space.
This is important for upper gradients and \p-weak upper gradients to make sense,
and is also convenient in other places, e.g.\ when defining the capacities below.

We say  that $u \in \Nploc(X)$ if
for every $x \in X$ there is a ball 
$B(x,r_x)$
 such that $u \in \Np(B(x,r_x))$. 
If $u,v \in \Nploc(X)$, then $g_u=g_v$ a.e.\ in $\{x \in X : u(x)=v(x)\}$.
In particular $g_{\min\{u,c\}}=g_u \chione_{\{u < c\}}$ a.e.\ in $X$
for $c \in \R$, where $\chione$ denotes the characteristic function.
For any nonempty open set $\Om \subset X$,
the spaces $\Np(\Om)$ and $\Nploc(\Om)$ are defined by
considering $(\Om,d|_\Om,\mu|_\Om)$ as a metric space in its own right.

The  \emph{Sobolev capacity} of an arbitrary set $E\subset X$ is
\[
\Cp(E) =\inf_{u}\|u\|_{\Np(X)}^p,
\]
where the infimum is taken over all $u \in \Np(X)$ such that
$u\geq 1$ on $E$.
The Sobolev capacity is countably subadditive.

A property of
points $x\in X$ holds \emph{quasieverywhere} (q.e.)\
if it is true outside a set of capacity  zero.
The Sobolev capacity is the correct gauge
for distinguishing between two Newtonian functions.
If $u \in \Np(X)$,  then $v \sim u$ if and only if $v=u$ q.e.
Moreover, if $u,v \in \Np(X)$ and $u= v$ a.e., then $u=v$ q.e.

The Newtonian space with zero boundary values is defined by
\[
\Np_0(\Om):=
  \{f|_{\Om} : f \in \Np(X) \text{ and }
        f=0 \text{ on } X \setm \Om\}.
\]

\begin{deff} \label{deff-cp}
Let $\Om\subset X$ be a (possibly unbounded) open set.
The \emph{condenser capacity} of a bounded set
$E \subset \Om$ with respect to 
$\Om$ is
\begin{equation*} 
\cp(E,\Om) = \inf_u\int_{\Om} g_u^p\, d\mu,
\end{equation*}
where the infimum is taken over all 
$u \in \Np_0(\Om)$ such that $u=1$ on $E$.
If no such function $u$ exists then $\cp(E,\Om)=\infty$.
If $E \subset \Om$ is unbounded, we define
(as in Bj\"orn--Bj\"orn~\cite{BBglobal})
\begin{equation*} 
\cp(E,\Om)=\lim_{j \to \infty} \cp(E \cap B_j,\Om).
\end{equation*}
\end{deff}

\begin{deff} \label{def-PI}
Let $q \ge 1$.
We say that $X$ or $\mu$ supports a \emph{$q$-Poincar\'e inequality} if
there exist constants $\CPI>0$ and $\lambda \ge 1$
such that for all balls $B=B(x,r)$, 
every integrable function $u$ on $X$, and all 
upper gradients $g$ of $u$ on $X$, 
\begin{equation} \label{eq-PI} 
       \vint_{B} |u-u_B| \,d\mu
        \le \CPI r \biggl( \vint_{\lambda B} g^{q} \,d\mu \biggr)^{1/q},
\end{equation}
where $ u_B 
 :=\vint_B u \,d\mu 
:= \mu(B)^{-1}\int_B u\, d\mu$.
\end{deff}

There are many equivalent formulations of Poincar\'e inequalities, see e.g.\
\cite[Proposition~4.13]{BBbook} and~\cite{HKSTbook}.
In particular, \eqref{eq-PI} can equivalently be required for all 
\p-weak upper gradients.
At the same time, in complete spaces with a doubling measure $\mu$, 
it suffices if~\eqref{eq-PI} 
holds for all compactly supported Lipschitz functions with compactly supported
Lipschitz upper gradients, see Keith~\cite{KeithMZ03}.

A \emph{weight} $w$ on $\R^n$ is a nonnegative
locally integrable function.
If $d\mu=w \,dx$ is a doubling measure, which
supports a  \p-Poincar\'e inequality on $\R^n$,
then $w$ is called a \emph{\p-admissible weight}.
See Corollary~20.9 in Heinonen--Kilpel\"ainen--Martio~\cite{HeKiMa} and
Proposition~A.17 in~\cite{BBbook}
for why this is equivalent to other definitions in the literature.
Moreover, in this case $g_u=|\nabla u|$ a.e.\ if $u \in \Np(\R^n)$,
where $\nabla u$ is the 
 gradient from~\cite{HeKiMa},
and the Sobolev and condenser capacities coincide with those 
in~\cite{HeKiMa}, see 
Theorem~6.7 and Proposition~A.12 in~\cite{BBbook} and 
Proposition~7.2 in 
Bj\"orn--Bj\"orn~\cite{BBglobal}.
In particular, $\R^n$ equipped with a \p-admissible weight is
a special case of the metric spaces considered in this paper
and all our results apply in that case.
A rich potential theory for such weighted $\R^n$ was
developed  in~\cite{HeKiMa}.
Also many manifolds and other spaces are included in our study.

Throughout the paper, we write $a \simle b$ and $b \simge a$ if there is an implicit
comparison constant $C>0$ such that $a \le Cb$, 
and $a \simeq b$ if $a \simle b \simle a$.
The implicit comparison constants are allowed to depend on the 
fixed data.

A \emph{domain} is a nonempty connected open set.
As usual, we write $u_+= \max\{u,0\}$,
and by $E \Subset \Om$ we mean that $\itoverline{E}$
is a compact subset of $\Om$.
In this  paper, a continuous function is always assumed
to be real-valued (as opposed to $\eR$-valued).

\section{\texorpdfstring{\p}{p}-harmonic and superharmonic functions}
\label{sect-pharm}

\emph{In addition to the assumptions from the beginning of Section~\ref{sect-ug},
  we assume from now on that
  $X$ is a complete  metric space
equipped with a doubling measure $\mu$ that
supports a \p-Poincar\'e inequality, where $1<p<\infty$.
We also fix a point $x_0 \in X$, write $B_r=B(x_0,r)$,
and assume that $\Om \subset X$ is a nonempty open set.
}

\medskip

It follows from these assumptions that $X$ is  proper
(i.e.\ every closed bounded set is compact) and connected.
Moreover, $X$ is \emph{quasiconvex}, i.e.\ there is a constant $C$
such that for every pair of points $x,y \in X$ there is a rectifiable curve $\ga$
from $x$ to $y$ with length $\ell_\ga \le C d(x,y)$.
In particular $X$ is locally connected, and components of open sets
are open.
See e.g.\ \cite[Proposition~3.1  and Theorem~4.32]{BBbook} or 
\cite[Lemma~4.1.14 and~Theorem~8.3.2]{HKSTbook}.

A function $u \in \Nploc(\Om)$ is a
\emph{minimizer} in $\Om$
if 
\begin{equation} \label{eq-def-pharm}
      \int_{\phi \ne 0} g^p_u \, d\mu
           \le \int_{\phi \ne 0} g_{u+\phi}^p \, d\mu
           \quad \text{for all } 
\phi \in \Np_0(\Om).
\end{equation} 
Equivalently,  one can require \eqref{eq-def-pharm} 
only for $\phi \in \Lipc(\Om)$,
see Bj\"orn~\cite[Proposition~3.2]{ABkellogg} 
(or \cite[Proposition~7.9]{BBbook}) for
this  
and other characterizations.
A continuous minimizer is called a \emph{\p-harmonic function}.
(It follows from Kinnunen--Shan\-mu\-ga\-lin\-gam~\cite{KiSh01} that
every minimizer has a continuous representative.)

\begin{deff} \label{deff-superharm-class}
A function $u : \Om \to (-\infty,\infty]$ is 
\emph{superharmonic} in $\Om$ if
\begin{enumerate}
\renewcommand{\theenumi}{\textup{(\roman{enumi})}}%
\item \label{cond-a} $u$ is lower semicontinuous;
\item \label{cond-b} 
 $u$ is not identically $\infty$ in any component of $\Om$;
\item \label{cond-c}
for every nonempty open set $G \Subset \Om$ with $\Cp(X \setm G)>0$
and every function $v\in C(\clG)$ which is \p-harmonic in $G$ and such
that $v\le u$ on $\bdy G$, we have $v\le u$ in $G$.
\end{enumerate}
\end{deff}

This definition of superharmonicity is the same as the one
usually used in the Euclidean literature, e.g.\ in
Hei\-no\-nen--Kil\-pe\-l\"ai\-nen--Martio~\cite[Section~7]{HeKiMa}.
It is equivalent to other definitions of superharmonicity on metric spaces,
by Theorem~6.1 in Bj\"orn~\cite{ABsuper}
(or \cite[Theorem~14.10]{BBbook}).
It is not difficult to see that a function $u$ is \p-harmonic if and only if 
both $u$ and $-u$ are superharmonic.

The Harnack
inequality for nonnegative \p-harmonic function was obtained 
in Kinnunen--Shan\-mu\-ga\-lin\-gam~\cite[Corollary~7.3]{KiSh01} using De Giorgi's method
(see however \cite[Section~10]{BMarola} or \cite[Section~8.4]{BBbook} for
some necessary modifications). 
A different proof, using Moser iteration, was given by 
Bj\"orn--Marola~\cite[Theorem~9.3]{BMarola},
and one using a combination of both methods by 
Bj\"orn--Bj\"orn~\cite[Theorem~8.12]{BBbook}. 
The strong maximum principle is a direct consequence of the Harnack inequality.

\begin{theorem} \label{thm-Harnack} 
\textup{(Harnack's inequality)}
There exists a constant  $A>0$ depending only on $p$ and 
the doubling and Poincar\'e constants $C_\mu$, $\CPI$ and $\la$,
such that for any \p-harmonic function $u\ge 0$ in $\Om$, one has
\[
\sup_{B} u \le A \inf_{B} u,
\]
for every ball $B \subset 50\la B \subset \Omega$. 
\end{theorem}

We need the following version of Harnack's inequality, which is obtained by iteration.

\begin{prop}\label{prop:Harnack} 
There are constants $C_0,\al>0$, depending only on $p$ and
the doubling and Poincar\'e constants $C_\mu$, $\CPI$ and $\la$,
such that if $u\ge0$ is \p-harmonic in a ball $B$, then for all $0<\de\le 1/50\la$,
\[
\sup_{\de B} u \le (1+C_0\de^\al)  \inf_{\de B} u.
\]
\end{prop}

\begin{proof}
Let $k$ be the smallest integer such that $\de> (50\la)^{-k-2}$. 
Note that $k\ge0$.  
For $j=0,1,\dots, k+1$, let 
\[
B^j=(50\la)^j \de B,   \quad M_j = \sup_{B^j} u \quad \text{and} 
\quad m_j = \inf_{B^j} u.
\]
Then $u-m_{j+1}$ is a nonnegative \p-harmonic function in 
$B^{j+1} = 50\la B^j \subset B$.
Let $A$ be the constant in Harnack's inequality (Theorem~\ref{thm-Harnack}).
Thus
\[
M_j-m_{j+1}  \le A(m_j-m_{j+1} )  
\]
and hence
\[
A(M_j-m_{j})  
 =(A-1)M_j + M_j -A m_j
\le
(A-1)(M_j-m_{j+1} ).  
\]
Dividing by $Am_{j+1}$, we get
\[
\frac{M_j}{m_j}  -1\le \frac{M_j-m_j}{m_{j+1}}   \le \frac{(A-1)(M_{j} - m_{j+1})}{Am_{j+1}} 
\le \frac{A-1}{A} \biggl( \frac{M_{j+1}}{m_{j+1}} - 1 \biggr).
\]
Iterating this inequality and another use of the Harnack inequality yields
\[
\frac{M_{0} }{m_0} - 1 
\le \biggl( \frac{A-1}{A} \biggr)^k \biggl( \frac{M_{k}}{m_{k}} - 1 \biggr)
\le \biggl( \frac{A-1}{A} \biggr)^{k+1} A \le C_0 \de^\al,
\]
where $\al \log 50\la = \log(A/(A-1))>0$ 
and $C_0= A(50\la)^{\al}$.
\end{proof}

The following lemma will serve as
a substitute for the Loewner-type estimate \cite[Lemmas~2.11 and~2.12]{BCZ} 
used in the proof of Lemma~3.4 in
Bonk--Capogna--Zhou~\cite{BCZ} for $Q$-harmonic functions
in Ahlfors $Q$-regular spaces.
Here we do not assume
 Ahlfors regularity and consider \p-harmonic functions for any $1<p<\infty$.

\begin{lem}  \label{lem-est-M-m-2B}
Let $u$ be  a superharmonic function in a domain $\Om$, which
is \p-harmonic in $\Om\setm \clB$  for some ball $B=B(x_0,r)$ such that 
$4B \subset \Om$.
Let
\[
M = \sup_{3B \setm 2B} u \quad \text{and} \quad m = \inf_{3B\setm 2B} u
\]
and assume that $u$ is normalized so that for all $m \le a<b\le M$,
\begin{equation}   \label{eq-normalized}
\int_{\{x\in\Om:a< u(x)<b\}} g_u^p \, d\mu = b-a.
\end{equation}
If $\Om_M:= \{x\in\Om:u(x)>M\} \Subset\Om$, then
\[
M-m  \simle \biggl( \frac{\mu(B)}{r^p} \biggr)^{1/(1-p)},
\]
where the comparison constant in $\simle$ only depends on $p$ and
the doubling and 
Poincar\'e constants
$C_\mu$, $\CPI$ and $\la$.
\end{lem}

\begin{proof}
By the lower semicontinuity and the strong
minimum principle for superharmonic functions 
\cite[Theorem~9.13]{BBbook},
we see that 
\[
m = \min_{\overline{3B}} 
u
\]
is attained at some $x_m\in \bdry (3B)$ and that 
$3B \subset \{x\in\Om:u(x) \ge m\}$.
 
Next, let $G$ be a component of $\Om_M$.  
Since $u\le M$ in $3B\setm 2B$, we have
either $G\subset 2B$ or $G\subset \Om\setm 3B$.
However, the latter is impossible because of the strong maximum principle for \p-harmonic functions
and the assumption that $G\Subset \Om$.
We therefore conclude that $\Om_M\subset 2B$ and
by the continuity of $u$ in $\Om\setm \clB$, there is $x_M\in \bdry (2B)$ such that $M= u(x_M)$.

Next, note that $u-m$ and  $M-u$ are nonnegative \p-harmonic functions in 
\[
B^M:=B(x_M,r)\subset 3B\setm B
\quad \text{and} \quad 
B^m:=B(x_m,r)\subset 4B\setm 2B,
\]
respectively.
Proposition~\ref{prop:Harnack} then provides us with 
$0<\de<1$,
depending only on 
$p$, $C_\mu$, $\CPI$ and $\la$,
so that 
\[
\sup_{\de B^M} (u-m) \le \frac43 \inf_{\de B^M} (u-m)
\quad \text{and} \quad
\sup_{\de B^m} (M-u) \le \frac43 \inf_{\de B^m} (M-u).
\]
Then 
\[
M' := \inf_{\de B^M} u \ge m + \frac34 (M-m)
\quad \text{and} \quad
m' := \sup_{\de B^m} u \le M - \frac34 (M-m).
\]
Subtracting the second inequality from the first, we get that
\begin{equation}   \label{eq-M-m-M'-m'}
M'-m' \ge \tfrac12 (M-m).
\end{equation}
Next, let $v=\min\{M',\max\{m',u\}\}$ be the truncation of $u$ at 
the levels $m'$ and~$M'$.
Then, by the doubling property of $\mu$,
\[
\vint_{4B} |v-v_{4B}| \, d\mu 
\ge \frac12 (M'-m') \frac{\min\{\mu(\de B^M),\mu(\de B^m)\}} {\mu(4B)}
\simge M'-m'.
\]
Inserting this into the \p-Poincar\'e inequality for $v$ in $4B$ implies that
\[
M'-m' \simle r \biggl( \vint_{4\la B}  g_v^p \, d\mu \biggr)^{1/p}
\simle \frac{r}{\mu(B)^{1/p}} \biggl( \int_{\{x\in\Om:m'< u(x)<M'\}}  g_u^p \, d\mu \biggr)^{1/p},
\]
where $\la$ is the dilation constant in the \p-Poincar\'e inequality.
Note that $m \le m' \le M' \le M$.
The normalization~\eqref{eq-normalized} 
then yields
\[
M'-m' \simle  \frac{r}{\mu(B)^{1/p}} (M'-m')^{1/p}
\]
and hence
\[
M'-m' \simle \biggl(  \frac{r^p}{\mu(B)} \biggr)^{1/(p-1)}.
\]
Finally, \eqref{eq-M-m-M'-m'} concludes the proof.
\end{proof}

\section{Singular and Green functions}
\label{sect-sing-Green}

Green functions are properly normalized singular functions, see 
Definition~\ref{deff-Green} and also Theorem~\ref{thm-Green} below
for the close relation between singular and Green functions.
We begin by defining singular functions.
The following definition was formulated in 
Bj\"orn--Bj\"orn--Lehrb\"ack~\cite{BBLgreen} for bounded $\Om$.

 \begin{deff} \label{deff-sing-bdd}
\textup{(\cite[Definition~1.1]{BBLgreen})}
 Let $\Om \subset X$ be a bounded domain. 
 A positive
 function $u:\Om \to (0,\infty] $ is a \emph{singular function} in $\Om$
 with singularity at $x_0 \in \Om$ if it satisfies the following
 properties\/\textup{:} 

 \addjustenumeratemargin{(S1)}
 \begin{enumerate}
 \renewcommand{\theenumi}{\textup{(S\arabic{enumi})}}%
  \item \label{dd-h}
 $u$ is \p-harmonic in $\Om \setm \{x_0\}$\textup{;}
\item \label{dd-s}
 $u$ is superharmonic in $\Om$\textup{;}  
 \item \label{dd-sup}
 $u(x_0)=\sup_\Om u$\textup{;}
 \item \label{dd-inf}
 $\inf_\Om u = 0$\textup{;}
 \item \label{dd-Np0}
 $\ut \in \Nploc(X \setm \{x_0\})$, where
 \[
    \ut = \begin{cases}
      u & \text{in } \Om, \\
      0 & \text{on } X \setm \Om.
      \end{cases}
 \]
 \end{enumerate}
\end{deff}
\medskip

The last condition~\ref{dd-Np0} says that $u$ has ``zero boundary values
in the Sobolev sense''.
This makes it possible to treat arbitrary bounded domains $\Om$ and not only 
regular bounded domains, where the zero boundary data are attained
as limits $\lim_{\Om\ni y\to x}u(y) =0$ for all $x\in\bdy\Om$.

However, Definition~\ref{deff-sing-bdd}
is not suitable for unbounded domains $\Om$, including $X$ itself, since 
the condition $\ut \in \Nploc(X \setm \{x_0\})$ 
does not capture the zero boundary condition at $\binfty$.
This is demonstrated by the following example.
To simplify the notation,  we use $\lim_{x \to \binfty}$ as a shorthand
for $\lim_{d(x,x_0)\to \infty}$.

\begin{example} 
Let $X=\R^n$, $n>2=p$ and $\Om= \R^n\setm \itoverline{B(0,1)}$. 
Let $u$ be the classical $2$-harmonic Green function
in $\Om$ with singularity at some $x_0 \in \Om$.
Then $\lim_{x \to \binfty} u(x)=0$ and $u$ satisfies 
Definition~\ref{deff-sing-bdd}.
Let $c>0$.
Then also the function $\ub(x):=u(x)+c(1-|x|^{2-n})$ satisfies the conditions in 
Definition~\ref{deff-sing-bdd}, but
$\lim_{x\to\binfty} \ub(x) = c>0$ is not desirable 
for Green functions, since $\R^n$ is $2$-hyperbolic
in the sense of Definition~\ref{def-p-par}.

Replacing $\ut \in \Nploc(X \setm \{x_0\})$ 
in Definition~\ref{deff-sing-bdd} by 
$\ut \in \Np(X \setm \clB_r)$ for every $r>0$
does not help since the desired Green function (such as $u(x)=c|x|^{(p-n)/(p-1)}$
in $\R^n$) is typically not $L^p$-integrable at $\binfty$.
\end{example}

Thus, for unbounded domains, the zero boundary value at $\binfty$ needs to be captured 
in a different way than by Newtonian spaces.
One way of doing this simultaneously both for finite boundary points 
and for $\binfty$ 
is as in Bj\"orn--Bj\"orn~\cite{BBglobal} by means of Perron solutions 
with respect to the extended boundary 
\[
\bdystar \Om:= \begin{cases}   
     \bdy\Om \cup \{\binfty\},   &  \text{if $\Om$ is unbounded,}  \\ 
     \bdy \Om,  & \text{otherwise,}  
\end{cases}
\]
where $\binfty$ is the point added in the \emph{one-point compactification} 
$\Xstar:=X \cup \{\binfty\}$ of~$X$ when $X$ is unbounded.

\begin{deff}\label{def:Perron}
Assume that $\bdystar \Om \ne \emptyset$.
  Given $f:\bdystar\Omega\to\eR$, 
let $\UU_f(\Omega)$ be the collection of all superharmonic functions 
$u$ in $\Omega$ that are bounded from below and such that 
\[
	\liminf_{\Omega\ni y\to x} u(y) \geq f(x)
	\quad\textup{for all }x\in\bdystar\Omega.
\]
The \emph{upper Perron solution} of $f$ is defined by 
\[
	\uP_\Omega f(x)
	= \inf_{u\in\UU_f(\Omega)} u(x),
	\quad x\in\Omega.
\]
The \emph{lower Perron solution} 
is 
given by 
$\lP_\Omega f:= - \uP_\Omega (-f)$.
If $\uP_\Omega f=\lP_\Omega f$, then we 
denote the common value by $P_\Omega f$.
\end{deff}

Perron solutions make 
it possible to define singular functions also in unbounded domains
as follows.

\begin{deff} \label{deff-sing-unbdd}
\textup{(\cite[Definition~9.1]{BBglobal})}
Let $\Om \subset X$ be a (possibly unbounded) domain. 
A positive function $u:\Om \to (0,\infty]$ 
is a \emph{singular function} in 
$\Om$
with singularity at $x_0 \in \Om$ if
it satisfies the following
properties\/\textup{:} 

\addjustenumeratemargin{(S1$'$)}
\begin{enumerate}
\renewcommand{\theenumi}{\textup{(S\arabic{enumi})}}%
\item \label{uu-h}
$u$ is \p-harmonic in $\Om \setm \{x_0\}$,
\item \label{uu-s}
$u$ is superharmonic in $\Om$,
\item \label{uu-sup}
$u(x_0)=\sup_\Om u$,
\item \label{uu-inf}
$\inf_\Om u = 0$, 
\renewcommand{\theenumi}{\textup{(S\arabic{enumi}$'$)}}%
\item \label{uu-P}
$  u=P_{\Om \setm \clB_r} \ut$ in $\Om \setm \clB_r$
whenever $B_r \Subset \Om$ and
 \[
    \ut = \begin{cases}
      u & \text{in } \Om, \\
      0 & \text{on } \Xstar \setm \Om.
      \end{cases}
 \]
\end{enumerate}
\end{deff}
\medskip

Note that it is only in the last conditions \ref{dd-Np0} and~\ref{uu-P} that 
Definitions~\ref{deff-sing-bdd}   and~\ref{deff-sing-unbdd} differ.
For bounded domains $\Om$, 
Definitions~\ref{deff-sing-bdd}  
 and~\ref{deff-sing-unbdd}
are equivalent by Proposition~9.2 in~\cite{BBglobal}.

\begin{deff} \label{deff-Green}
A \emph{Green function} in a domain $\Om$ 
  is a singular function that is normalized so that
\begin{equation} \label{eq-normalized-Green-intro-deff}
\cp(\{u  \ge b\},\Om) = b^{1-p},
\quad \text{when }
   0  <b < u(x_0).
\end{equation}
\end{deff}

In fact, it follows  from~\cite[Theorem~1.2]{BBglobal}
and~\cite[Theorem~9.3]{BBLgreen}
that Green functions $u$ satisfy 
\begin{equation} \label{eq-normalized-Green-intro}
\cp(\{u \ge b\},\{u >a\}) = (b-a)^{1-p},
\quad \text{when }
   0 \le a <b \le  u(x_0),
\end{equation}
where 
we interpret $\infty^{1-p}$ as $0$.
Here, we use the shorthand notation $\{u \ge b\}=\{x \in\Om:u(x) \ge b\}$ 
and similarly for other functions and constants.
The normalizations~\eqref{eq-normalized-Green-intro-deff} 
and~\eqref{eq-normalized-Green-intro} reflect the fact that the Green function 
in $\R^n$ satisfies the \p-Laplace equation~\eqref{p-Lap} with the
Dirac measure $\de_{x_0}$.

The following theorem   summarizes some of the main results 
in~\cite{BBglobal} and~\cite{BBLgreen}, 
which will be used in the sequel.

\begin{thm} \label{thm-Green}
\textup{(\cite[Theorems~1.1 and~1.2]{BBglobal} and~\cite[Theorem~1.3]{BBLgreen})}
\begin{enumerate}
\item \label{item-Green-a}
If $v$ is a singular function in $\Om$ 
with singularity at $x_0 \in \Om$ then there is a unique 
$A>0$ such that $A v$ is a Green function in $\Om$.
\item \label{item-Green-exist}
There is a Green function in a domain $\Om$ with singularity at $x_0 \in \Om$
if and only if either $\Cp(X \setm \Om)>0$ or $X$ is \p-hyperbolic, 
according to Definition~\ref{def-p-par} below.
\end{enumerate}
\end{thm}

\begin{deff} \label{def-p-par}
Assume that $X$ is unbounded.
Then $X$ is called \emph{\p-hyperbolic}
if $\cp(K,X)>0$ for some compact set $K\subset X$.
Otherwise, $X$ is \emph{\p-parabolic}. 
\end{deff}

If $X$ is \p-hyperbolic and $u$ is a Green function in
an unbounded domain $\Om$, then 
(under our standing assumptions)
\begin{equation} \label{eq-lim-infty}
\lim_{\Om \ni x \to \binfty} u(x)=0,
\end{equation}
by Lemma~11.1 in~\cite{BBglobal}.
The following example shows that this is not always true  in 
\p-parabolic spaces,
see  also Proposition~\ref{prop-parabolic} below.
This reflects the fact that in \p-parabolic spaces, the point at $\binfty$
has zero \p-capacity in a generalized sense.

Thus, \eqref{eq-lim-infty} cannot 
be used to capture the zero condition at infinity provided by \ref{uu-P},
which similarly as for finite boundary points only requires the boundary
value zero in a weaker sense.

\begin{example}
Let $X=\R^2$, $p=2$, $\Om= \R^2\setm \itoverline{B(0,1)}$ and $z_0 =2$ 
(using complex notation).
Then the Green function for $\Om$ with singularity at $z_0$ can be obtained by 
a M\"obius transformation
from the Green function in the unit disc $B(0,1)$ with
singularity at $0$
and can be calculated to be 
\[
  u(z) = A \log {}\biggl|\frac{2z-1}{2-z}\biggr|
\]
for some multiplicative constant $A>0$.
Then $\lim_{z\to\binfty} u(z) = A \log2>0$. 
Further, \ref{uu-P} is still true for this function,
since $\binfty$ has $2$-capacity zero 
in a generalized sense, and thus has harmonic measure zero.
This example is not just a coincidence, 
but illustrates a more general  phenomenon.
\end{example}

The following result shows that for unbounded sets in \p-parabolic spaces, 
we cannot require that Green functions tend to zero at $\binfty$.

\begin{prop} \label{prop-parabolic}
Assume that $X$ is \p-parabolic and that $K$
is a compact set with $\Cp(K)>0$, 
such that $\Om=X\setm K$ is a domain.
Let $u$ be a 
Green function in $\Om$ with singularity at~$x_0\in \Om$,
which exists by Theorem~\ref{thm-Green}\ref{item-Green-exist}.
Then 
\[
    \liminf_{x\to\binfty} u(x) >0.
\]
\end{prop}

\begin{proof}
Let $R>0$ be so large that $K\Subset B_R$ and let
$m= \min_{\bdy B_R} u>0$.
Consider the open set $G=X\setm \clB_R$ and the continuous function 
$f$ on $\bdystar G$ given by 
$f(x) =   u(x)$ for $x\in \bdy B_R$ and $f(\binfty)=m$.
Note that $G$, being a subset of the \p-parabolic space $X$, is a \p-parabolic set in the 
sense of Definition~4.1 in Hansevi~\cite{HanseviPerron}.
Since $u$ is \p-harmonic in $G$ and attains the boundary values $f$ on $\bdy G$,
Corollary~7.9 in~\cite{HanseviPerron} implies that $u=P_G f$ and hence 
$u\ge m>0$ in $G$.
\end{proof}

\section{Uniqueness of Green functions}

\label{sect-unique}

In this section, we prove Theorem~\ref{thm-unique-Green-intro} when $\Cp(\{x_0\})=0$. 
The main step in the proof is captured by the following lemma.

\begin{lem}   \label{lem-unique-Green}
Let $u$ be a Green function in a domain $\Om$ with singularity
at~$x_0$.
Assume that $\Cp(\{x_0\})=0$ and that there is a sequence $r_j \searrow 0$ such that
\begin{equation}  \label{eq-ass-for-unique}
\lim_{j\to\infty} \biggl( \frac{\mu(B_{r_j})}{r_j^p} \biggr)^{1/(p-1)}  
\int_{r_j}^1 \biggl( \frac{\rho^p}{\mu(B_\rho)} \biggr)^{1/(p-1)} \frac{d\rho}{\rho} = \infty.
\end{equation}
For $r>0$, let
\[
m(r) = \min_{d(x,x_0)=r} u(x)
\quad \text{and} \quad 
M(r) = \max_{d(x,x_0)=r} u(x).
\]
Then for every ball $B_R \Subset \Om$,
\begin{equation}   \label{eq-lim-M(r)/cap}
\lim_{j\to\infty} \frac{m(r_j)}{\cp(B_{r_j},B_R)^{1/(1-p)}} 
= \lim_{j\to\infty} \frac{M(r_j)}{\cp(B_{r_j},B_R)^{1/(1-p)}} = 1.
\end{equation}
\end{lem}

\begin{proof}
Since $\Cp(\{x_0\})=0$, we have $u(x_0)=\lim_{x\to x_0} u(x) =\infty$,
by Lemma~9.5 in Bj\"orn--Bj\"orn~\cite{BBglobal}.
Fix $R>0$ and let $r>0$ be sufficiently small so that $m(r)>M(R)$. 
By the continuity of $u$ and the
maximum principle for \p-harmonic functions
we see that 
\[
B_r \subset \{u\ge m(r)\}.
\]
Next, Lemma~9.9 in~\cite{BBglobal}
shows that the set $\{x\in\Om: u(x)>M(r)\}$
is connected and hence (using also that $u(x_0)=\infty$)
\[
\{u>M(r)\} \subset B_r.
\]
Similar statements hold with $r$ replaced by $R$ and 
we get
\[
\{u>M(r)\} \subset B_r \subset \{u\ge m(r)\} 
\subset \{u>M(R)\} \subset B_R \subset \{u \ge m(R)\}.
\]
Since $u$ is a Green function (and thus normalized), 
we then have by 
\eqref{eq-normalized-Green-intro} and
the monotonicity of the capacity that
\[  
m(r) - M(R) 
= \cp(\{u \ge m(r)\}, \{u>M(R)\})^{1/(1-p)} 
\le \cp(B_{r},B_R)^{1/(1-p)}
\]  
and 
\begin{align*}
M(r) - m(R) 
&= \lim_{\eps \to 0+} \cp(\{u\ge M(r)+\eps\}, \{u>m(R)-\eps\})^{1/(1-p)} \\
&\ge \cp(B_{r},B_R)^{1/(1-p)}.
\end{align*}
Hence
\[
\frac{m(r)}{M(r)} - \frac{M(R)}{M(r)}
\le \frac{\cp(B_{r},B_R)^{1/(1-p)}} {M(r)} \le 1 - \frac{m(R)}{M(r)}.
\]

Since $R$ is fixed and $M(r)\to\infty$ as $r\to0$, we have
\[
\lim_{r\to0} \frac{M(R)}{M(r)} = \lim_{r\to0} \frac{m(R)}{M(r)} = 0.
\]
It therefore suffices to show that $m(r_j)/M(r_j)\to1$, or equivalently
that 
\[
 \frac{M(r_j)-m(r_j)}{M(r_j)} \to0
  \quad \text{as } j\to\infty.
\]
To this end, Lemma~\ref{lem-est-M-m-2B} together with the estimate for $M(r)$ 
from 
Bj\"orn--Bj\"orn--Lehrb\"ack~\cite[Theorem~7.1]{BBLehIntGreen} 
and the doubling property of $\mu$ give 
that, for sufficiently small $r_j$,
\begin{equation*} 
0 \le \frac{M(r_j)-m(r_j)}{M(r_j)} 
\simle 
\biggl( \frac{\mu(B_{r_j})}{r_j^p} \biggr)^{1/(1-p)}
\biggl(\int_{r_j}^1 \biggl( \frac{\rho^p}{\mu(B_\rho)} \biggr)^{1/(p-1)} \frac{d\rho}{\rho} \biggr)^{-1}
\to 0,
\end{equation*} 
as $j \to \infty$,
by the assumption~\eqref{eq-ass-for-unique}.
This proves~\eqref{eq-lim-M(r)/cap}.
\end{proof}

\begin{cor}  \label{cor-unique-Green-1}
Let $u_k$  be a Green function in a 
domain $\Om_k \ni x_0$ with singularity at~$x_0$, $k=1,2$.
Assume that $\Cp(\{x_0\})=0$ and that 
\eqref{eq-ass-for-unique-limsup} holds.
Then
\[
    \lim_{x \to x_0} \frac{u_1(x)}{u_2(x)}=1.
\]
\end{cor}

\begin{proof}
Choose $r_j \searrow 0$ such that \eqref{eq-ass-for-unique} holds.
 Lemma~\ref{lem-unique-Green}  implies that for each $\eps>0$, 
there is $j_0$ such that for all $j\ge j_0$,
\[
\max_{d(x,x_0)=r_j} u_1 
\le (1+\eps) \cp(B_{r_j},B_R)^{1/(1-p)}
\le (1+\eps)^2 \min_{d(x,x_0)=r_j} u_2.
\]
The maximum principle applied to the annuli $B_{r_j}\setm \clB_{r_{j+1}}$ 
 then gives that 
$u_1 \le (1+\eps)^2 u_2$ 
also in 
$B_{r_j}\setm \clB_{r_{j+1}}$for every 
$j \ge j_0$, 
and hence  in $B_{r_j}$.
Letting $\eps \to 0$ and applying this also with the roles
of $u_1$ and $u_2$ interchanged completes the proof.
\end{proof}

\begin{proof}[Proof of Theorem~\ref{thm-unique-Green-intro}]
When $\Cp(\{x_0\})>0$, the Green function is a particular multiple 
of the capacitary potential for $\cp(\{x_0\},\Om)$ and is therefore unique, see
Bj\"orn--Bj\"orn--Lehrb\"ack~\cite[Theorem~1.3]{BBLgreen}
(for bounded $\Om$) and
Bj\"orn--Bj\"orn~\cite[Corollary~10.3]{BBglobal}
(for general $\Om$).   
Assume therefore
that $\Cp(\{x_0\})=0$.
Let $\eps >0$.
By Corollary~\ref{cor-unique-Green-1}, there is $r<\eps$ such that 
\[
     u \le (1+\eps) v 
  \quad \text{in } B_{2r} \setm \{x_0\}.
\]
Since also $u=v=0$ on $\bdystar \Om$,
it follows from the definitions of singular functions and Perron solutions
that
\[
     u 
   = P_{\Om \setm \clB_r} u 
   \le (1+\eps) P_{\Om \setm \clB_r} v
   = (1+\eps) v
\quad \text{in } \Om \setm \clB_r.
\]
Hence $u \le (1+\eps) v$ in $\Om$.
Letting $\eps \to 0$ and applying this also with the roles
of $u$ and $v$ interchanged shows that $u=v$ in $\Om\setm \{x_0\}$,
while $u(x_0)=v(x_0)=\infty$ follows from $\Cp(\{x_0\})=0$.
\end{proof}

Condition~\eqref{eq-ass-for-unique-limsup}
can be determined by the following exponent sets:
\begin{equation} \label{eq-lQo}
\lQ_0 := \biggl\{   q>0:  
\frac{\mu(B_{r})}{\mu(B_{R})}   \simle  \Bigl( \frac{r}{R} \Bigr)^q  \text{ for all }
0<r<R\le1 \biggr\},
\end{equation}
and 
\begin{equation} \label{eq-uQo}
\uQ_0 := \biggl\{   q>0:    
\frac{\mu(B_{r})}{\mu(B_{R})} \simge \Bigl( \frac{r}{R} \Bigr)^q  \text{ for all }
0<r<R\le1 \biggr\}.
\end{equation}

Note that $\lQ_0$ and $\uQ_0$ are nonempty intervals, 
under our standing assumptions.
The general case when condition~\eqref{eq-ass-for-unique-limsup} 
holds in Remark~\ref{mainthm-remark-intro} follows from 
Proposition~\ref{prop-uQ-lQ}\ref{k-a} below.

\begin{prop} \label{prop-uQ-lQ}
$\quad$
\begin{enumerate}
\item \label{k-a}
If $p\notin \lQ_0$ or $p\in \uQ_0$, 
then 
\eqref{eq-ass-for-unique-limsup}
 holds.
\item \label{k-b}
If  $p <\sup \lQ_0$,
then 
\eqref{eq-ass-for-unique-limsup} fails.
\end{enumerate}
\end{prop}

\begin{proof}
\ref{k-a}
If $p\notin \lQ_0$, it follows from the definition of $\lQ_0$ that there are
$0<r_j<R_j \searrow0$ such that for each $j=1,2,\dots$\,,
\[
\frac{\mu(B_{r_j})}{\mu(B_{R_j})} \ge j \biggl( \frac{r_j}{R_j} \biggr)^p.
\]
By passing to a subsequence, we can assume that also $r_j\searrow0$
and that $R_1\le\tfrac12$.
Now, by the doubling property of $\mu$ we have for each $j$,
\[  
\int_{r_j}^1 \biggl( \frac{\rho^p}{\mu(B_\rho)} \biggr)^{1/(p-1)} \frac{d\rho}{\rho} 
\ge \int_{R_j}^{2R_j} \biggl( \frac{\rho^p}{\mu(B_\rho)} \biggr)^{1/(p-1)} \frac{d\rho}{\rho} 
\simeq \biggl( \frac{R_j^p}{\mu(B_{R_j})} \biggr)^{1/(p-1)}.
\]  
Hence
\[   
\biggl( \frac{\mu(B_{r_j})}{r_j^p} \biggr)^{1/(p-1)}  
\int_{r_j}^1 \biggl( \frac{\rho^p}{\mu(B_\rho)} \biggr)^{1/(p-1)} \frac{d\rho}{\rho} 
\simge
 \biggl( \frac{\mu(B_{r_j})}{\mu(B_{R_j})} \frac{R_j^p}{r_j^p} \biggr)^{1/(p-1)} 
\ge j^{1/(p-1)}.
\]   
Letting $j\to\infty$ concludes the proof when $p\notin \lQ_0$.

If $p\in \uQ_0$, then for any $0<r<1$ by the definition of $\uQ_0$,
\[  
\int_{r}^1 \biggl( \frac{\rho^p}{\mu(B_\rho)} \biggr)^{1/(p-1)} \frac{d\rho}{\rho} 
\simge \int_{r}^{1} \biggl( \frac{r^p}{\mu(B_r)} \biggr)^{1/(p-1)} \frac{d\rho}{\rho} 
= \biggl( \frac{r^p}{\mu(B_{r})} \biggr)^{1/(p-1)} \log \frac1r 
\]  
and the rest of the argument is the same as for $p\notin \lQ_0$, 
upon letting $r\to0$.

\ref{k-b}
If $p <\sup \lQ_0$, then
there is $q \in \lQ_0$ with $q>p$.
Let $0<r<\frac12$.
Then
\[
   \frac{\mu(B_r)}{r^q}
   \simle \frac{\mu(B_\rho)}{\rho^q}
\quad \text{for } r<\rho <1,
\]
and thus
\begin{align*}
& \biggl( \frac{\mu(B_{r})}{r^p} \biggr)^{1/(p-1)}  
\int_{r}^1 \biggl( \frac{\rho^p}{\mu(B_\rho)} \biggr)^{1/(p-1)} \frac{d\rho}{\rho} \\
& \qquad \qquad \qquad 
\simle
r^{(q-p)/(p-1)} \int_{r}^1 \biggl( \frac{\mu(B_{\rho})}{\rho^q} 
\frac{\rho^p}{\mu(B_\rho)} \biggr)^{1/(p-1)} \frac{d\rho}{\rho} 
\simeq   1.
\qedhere
\end{align*}
\end{proof}

The following example is based on Example~3.3 in 
Bj\"orn--Bj\"orn--Lehrb\"ack~\cite{BBLringcap} and shows that
the range of $p$'s for which \eqref{eq-ass-for-unique} holds 
and $\Cp(\{0\})=0$ can be large.

\begin{example} \label{ex-abcd-alt}
Fix $1<a<b<c<d$.
We are now going to construct a weight $w$ on $\R^n$ so that
for $a<p\le c$ we have uniqueness 
in Theorem~\ref{thm-unique-Green-intro} with $x_0=0$ and at the same
time the associated capacity $\Cpw(\{0\})=0$, 
thus showing that this range can be a nondegenerate interval.

Let
\[
    \la =\frac{(c-a)(d-b)}{(b-a)(d-c)}
\]
and 
\[
   \alp_k = 2^{-\la^k} \text{ and } 
\be_k = \al_k^{(d-b)/(d-c)} = \al_{k+1}^{(b-a)/(c-a)},
 \quad  k=0,1,2,\dots.
\]
Note that $\la>1$ and thus $\al_k\to0$ as $k\to\infty$.
Also, $\al_{k+1}\ll\be_k\ll\al_k$.
Then the weight $w(x)=w(|x|)$ given by 
\[
      w(\rho)=\begin{cases}
             \be_{k}^{c-a}\rho^{a-n} = \alp_{k+1}^{b-a}\rho^{a-n}, 
  & \text{if } \alp_{k+1} \le \rho \le \be_k, 
                     \ k=0,1,2,\dots,\\
             \be_{k}^{c-d}\rho^{d-n} = \alp_{k}^{b-d}\rho^{d-n}, 
  & \text{if } \be_k \le \rho \le \alp_{k}, 
                     \ k=0,1,2,\dots,\\
                     \alp_0, & \text{if } \rho \ge \alp_0,
        \end{cases}
\] 
is continuous and $1$-admissible on $\R^n$, by 
Theorem~10.5 in Bj\"orn--Bj\"orn--Lehrb\"ack~\cite{BBLringcap}.
Moreover by Example~3.4 in~\cite{BBLringcap}, we have
\begin{equation} \label{eq-ex-abcd-alt}
\lQo=(0,a]   
    \quad \text{and} \quad
\uQo=[d,\infty).
\end{equation}

Let $\Om \ni 0$ be a domain such that $\Cpw(\Rn \setm \Om)>0$ or $p<n$
(which guarantees that the Green function below exists).
Theorem~\ref{thm-unique-Green-intro} 
and Remark~\ref{mainthm-remark-intro} then show that
the Green function 
with respect to the measure
$d\mu=w\,dx$ in $\Om$
with singularity at $0$ is unique whenever $p>a$.
Next, note that $\al_{k+1} \le \tfrac12 \be_k$ for sufficiently large $k$.
Since $\mu$ is a doubling measure, a direct calculation 
shows that
\[
\mu(B_{\be_k}) \simeq
\mu(\tfrac12 B_{\be_k})  
\simeq \mu(B_{\be_k} \setm \tfrac12 B_{\be_k}) 
\simeq \be_k^{c-a}  \int_{\frac12\be_k}^{\be_k} \rho^{a-1} \,d\rho
\simeq \be_k^c.
\]
If $p=c$, the general capacity estimate 
(1.7) in  Bj\"orn--Bj\"orn--Lehrb\"ack~\cite{BBLehIntGreen}
then implies that in the weighted space $(\R^n,w\,dx)$,
\begin{align*}
C_{c,w}(\{0\}) 
&\simeq \biggl( \int_0^1 \biggl( \frac{\rho^c}{\mu(B_\rho)} \biggr)^{1/(p-1)} \frac{d\rho}{\rho} \biggr)^{1-p}  \\
&\le \biggl(  \sum_{k=1}^\infty   \biggl( \frac{\be_k^c}{\mu(B_{\be_k})} \biggr)^{1/(p-1)}  
\int_{\frac12\be_k}^{\be_k} \frac{d\rho}{\rho} \biggr)^{1-p} = 0.
\end{align*}
It follows that $\Cpw(\{0\})=0$ for $p\le c$.
In particular, we have uniqueness and at the same time
$\Cpw(\{0\})=0$ (i.e.\ the Green function $u$ is unbounded and  $u(x_0)=\infty$) 
whenever $a<p \le c$.
At the same time, when $p>c$, uniqueness still holds since  $\Cpw(\{0\})>0$
by Proposition~8.2 in~\cite{BBLringcap}.

We can make two modifications in this example:
\begin{enumerate}
\renewcommand{\theenumi}{\textup{(\roman{enumi})}}%
\item Let $\wt(x)$ be a bounded measurable function on $\R^n$ such that
$\inf_{\R^n} \wt >0$, and let $w_2(x)=w(x)\wt(x)$.
Then $w_2$ is also $1$-admissible and has the same 
exponent sets~\eqref{eq-ex-abcd-alt} as $w$,
so we get the same conclusions also for this nonradial weight. 

\item We can moreover equip $X=\R^n$ with a different metric,
e.g.\ an $\ell^q$ metric with $1 \le q  \le \infty$. 
As long as it is biLipschitz equivalent to the Euclidean metric,
the measure will still be globally doubling and supporting a global
$1$-Poincar\'e inequality on $X$, by \cite[Proposition~4.16]{BBbook},
and the exponent sets~\eqref{eq-ex-abcd-alt} 
will remain the same.
So the conclusion is the same also in this case.

\end{enumerate}

\end{example}

\section{Nonuniqueness of Green functions}
\label{sect-nonuniq}

In the following example we construct a nonstandard metric on $\R^2$ so
that Green functions with singularity at the origin
are not unique, both in $\R^2$ and in discs $\{x\in\R^2: |x|<R\}$.
This proves Theorem~\ref{thm-nonunique}.
The metric is comparable to the Euclidean metric 
and possesses a Finsler type structure. 

\begin{example} \label{ex-Finsler-1}
{\rm(}\emph{Proof of Theorem~\ref{thm-nonunique}}{\rm)}
Let $X=\R^2$ be equipped with the weighted Lebesgue measure
$d\mu = w \, dx$, where $w(x)=|x|^\alp$ with $\alp >-1$,
and a nonstandard metric $d$ that we will now define.
We also let $1<p<2+\alp$.
Note that the weight $w$ is \p-admissible, see
Hei\-no\-nen--Kil\-pe\-l\"ai\-nen--Martio~\cite[Corollary~15.35]{HeKiMa}. 

For $z \in \R^2 \setm \{0\}$ we use polar coordinates  with
complex notation and
write $z=re^{i\theta}$   with $r>0$ and  
$\theta \in \R$.
Thus our functions will be $2\pi$-periodic in $\theta$.
Equivalently, we restrict $\theta$ to the unit circle
$\Spe=[0,2\pi)$ (identifying $0$ and $2\pi$).
In polar coordinates, the standard orthonormal basis  for the tangent space $T_z \R^2$ 
at $z=re^{i\theta}\ne 0$  is given by 
\[
\partial_r:=
\frac{\partial}{\partial r}
\quad \text{and} \quad
\frac{1}{r}\partial_\theta
:=\frac{1}{r}\frac{\partial}{\partial \theta}.
\]
The metric $d$ on $\R^2$ is defined  using an $\ell^1$-norm on $T_z \R^2$ by 
\[
d(x,y)=\inf_{\gamma_{x,y}}\int_0^1 \|\gadot(t)\|_{\ga(t)}\ dt,
\quad x,y \in \R^2,
\]
where the infimum is taken over all Lipschitz curves $\ga_{x,y}$ 
connecting $x$ and $y$,
and
\begin{equation*} 
 \|\gadot(t)\|_{z}=|a|+|b|
 \quad \text{when } z \ne 0 \text{ and }
 \gadot(t)= a \partial_r + \frac{b}{r}\partial_\theta,
\end{equation*}
while $\|\gadot(t)\|_{0}=0$.

The metric $d$ on $\R^2$
is comparable to the Euclidean metric by a factor of $\sqrt{2}$, 
since the Euclidean norm of $\gadot$ is given by $\sqrt{a^2+b^2}$. 
It follows that
the measure~$\mu$ is doubling and supports a \p-Poincar\'e inequality
also on $(\R^2,d)$, by e.g.\ \cite[Proposition~4.16]{BBbook}.

Let $u\in \Liploc(\R^2\setminus\{0\})$. 
By Rademacher's theorem (see e.g.\ \cite[Theorem~3.2]{Evans}) 
for a.e.\ $z\in \R^2$ 
there exists a linear mapping (the differential) $du(z):\R^2\to\R^2$ such that
\[
\lim_{h\to 0}\frac{|u(z+h)-u(z)-du(z)(h)|}{\|h\|_z}=0.
\]
Expressed in polar coordinates with $h= a\bdy_r + b \frac{1}{r}\bdy_\theta$, 
the differential can be written as
\[
du(z)(h) = a\bdy_r u+ b \frac{1}{r}\bdy_\theta u.
\]
Using this
and Theorem~6.1 in Cheeger~\cite{Che99} (or Theorem~13.5.1
in Heinonen--Koskela--Shan\-mu\-ga\-lin\-gam--Tyson~\cite{HKSTbook}),
we obtain that for any $u\in \Liploc(\R^2\setminus\{0\})$ and
a.e.\ $z=re^{i\theta}\in\R^2$,
\begin{align*}
 g_u(z)&=\Lip u(z) 
 := \limsup_{\rho\to0} \sup_{\|h\|_z\le \rho} 
\frac{|u(z+h)-u(z)|}{\|h\|_z} \\
&= \lim_{\rho\to 0} \sup_{|a|+|b|\leq \rho}\frac{|a\bdy_r u(z) + b \frac{1}{r}\bdy_\theta u(z)|}{|a|+|b|} 
= \max\biggl\{|\bdy_r u(z)|,\frac{|\bdy_\theta u(z)|}{r}\biggr\}.
\end{align*}
This calculation expresses the fact that the dual basis of
$\{\bdy_r, \tfrac{1}{r} \bdy_r\}$ is given by $\{dr, r\,d\theta\}$, 
and the dual of the $\ell^1$-norm on the tangent space is the $\ell^\infty$-norm 
on the cotangent space.

Next, we let $f:\Spe\to \R$ be a $1$-Lipschitz function and 
\begin{equation}   \label{eq-def-ap}
a_p=\frac{2+\alp-p}{p-1}>0.
\end{equation}
Recall that we have assumed that $1<p<2+\alp$ and that $\alp > -1$.
Let $0 < R < \infty$ be fixed and define
\[
u(re^{i\theta})= (r^{-a_p}-R^{-a_p})e^{a_p f(\theta)}
\quad \text{for } 0 <r < R \text{ and } \theta \in \Spe.
\]
Also let $u(0)=\infty$.
We shall show that $u$ is a singular function in $\Om=\{x \in \R^2:|x| < R\}$
with singularity at $x_0=0$, with respect to $(\R^2,d,\mu)$.

First, we show that $u$ is \p-harmonic in $\Om\setm\{0\}$.
To this end, we have
\[
|\partial_r u(re^{i\theta})| = a_p r^{-a_p-1} e^{a_p f(\theta)}
\]
and, since $|f'(\theta)|\le1$ for a.e.\ $\theta \in \Spe$,
\[
\frac{|\partial_\theta u(re^{i\theta})|}{r} 
   = a_p \frac{r^{-a_p}-R^{-a_p}}{r}  e^{a_p f(\theta)} |f'(\theta)|
    \le |\partial_r u(re^{i\theta})|
\quad \text{a.e.}
\]
Hence  
\[
g_u(re^{i\theta}) = |\partial_r  u(re^{i\theta})| 
= a_p r^{-a_p-1} e^{a_p f(\theta)}
\quad \text{a.e.}
\]

Let  $0<a<b<R$ and $\phi \in \Lipc(G_{a,b})$, where $G_{a,b}=\{x : a<|x|<b\})$.
Fix $\theta \in \Spe$ for the moment.
A straightforward calculation shows 
that the function $r\mapsto u(re^{i\theta})$ 
is a \p-harmonic function on $(0,\infty)$ 
with respect to  the measure $r^{1+\alp}\,dr$.
Indeed, the corresponding Euler--Lagrange equation is
$-(|u'|^{p-2}u' r^{1+\alp})'=0$, see 
Hei\-no\-nen--Kil\-pe\-l\"ai\-nen--Martio~\cite[p.~102]{HeKiMa}.
(They consider $\R^n$ with $n \ge 2$, but 
the same calculation applies for $n=1$.)

Thus, comparing the functions $r\mapsto u(re^{i\theta})$ and 
$r\mapsto (u+\varphi)(re^{i\theta})$ on the interval $(a,b)$, we get
\begin{align*}
   \int_a^b g_u(re^{i\theta})^p r^{1+\alp} \, dr
  &= \int_a^b |\partial_r u(re^{i\theta})|^p r^{1+\alp}\,dr \\
  &\le \int_a^b |\partial_r (u+\phi)(re^{i\theta})|^p r^{1+\alp}\,dr 
  \le   \int_a^b g_{u+\varphi}(re^{i\theta})^p r^{1+\alp}\,dr.
\end{align*}
Integrating over $\theta \in \Spe$  yields
\[
\int_{G_{a,b}} g_u^p \,dx 
   \leq \int_{G_{a,b}} g_{u+\varphi}^p \, dx.
\]
Hence $u$ is \p-harmonic in $G_{a,b}$ for every $0<a<b<R$ and thus in 
$\Om\setm\{0\}$.

Since $u(0)=\infty$ and 
$\lim_{G \ni x \to z} u(x)=0$ 
whenever $|z|=R$, 
it follows from
Theorem~7.2 in
Bj\"orn--Bj\"orn--Lehrb\"ack~\cite{BBLgreen} 
that $u$ is a singular function in $\Om$.
Thus by Theorem~\ref{thm-Green}\ref{item-Green-a}, 
there is $A>0$  
such that $Au$ is a Green function in  $\Om$ with singularity at~$0$.
Since $f$ was an arbitrary $1$-Lipschitz function on  $\Spe$,
this shows the nonuniqueness of Green functions in  $\Om$
with singularity at $0$.

If we instead let
\[
u(re^{i\theta})= r^{-a_p}e^{a_p f(\theta)}
\quad \text{for }   r>0 \text{ and } \theta \in \Spe,
\]
with  $u(0)=\infty$,
then the same calculation shows that $u$ is \p-harmonic in $\R^2 \setm \{0\}$.
Since $\Cp(\{0\})=0$, it follows from Lemma~4.3 in~\cite{BBLgreen} 
that $u$ is superharmonic in $\R^2$,
and thus, by Definition~\ref{deff-sing-unbdd},
$u$ is a singular function in $X=(\R^2,d,\mu)$ with singularity at $0$.
Therefore,
by Theorem~\ref{thm-Green}\ref{item-Green-a}, 
there is $A>0$  
such that $Au$ is a Green function in  $X$ with singularity at $0$.
Again, since $f$ was an arbitrary $1$-Lipschitz function on  $\Spe$,
the nonuniqueness follows
also in this case.

Finally, since  the exponent sets 
in \eqref{eq-lQo} and~\eqref{eq-uQo}
for the measure $\mu$ are
\begin{equation*} 
\lQo=(0,2+\alp]   
    \quad \text{and} \quad
\uQo=[2+\alp,\infty),
\end{equation*}
the Green function  with singularity at $0$ is
unique for $p \ge 2+\alp$ in every domain $\Om \ni 0$
in $(\R^2,d,\mu)$, by Theorem~\ref{thm-unique-Green-intro}
and Remark~\ref{mainthm-remark-intro}.

Note that $(\R^2,d,\mu)$ is \p-hyperbolic if and only if $p< 2+\alp$,
while $\Cp(\{0\})=0$ if and only if $p \le 2+\alp$.
Note also
that in the unweighted case (with $\alp=0$) the range when 
the Green functions are not unique  is $1<p<2$.
\end{example}

\begin{remark}
One can make similar examples in unweighted and weighted $\R^n$, $n \ge 3$.
Since the circle $\Spe$ has to be replaced by the $(n-1)$-dimensional 
sphere $\Spn$, the norm on $T_z\R^n$ is for $z=r\om$ and $h\in\R^n$, 
with $\om\in \Sp^{n-1}$ and $r>0$, given by 
\[
\|h\|_{z} = |h_r| + |h_{\Sp^{n-1}}|, 
\]
where $h_r=\frac{h\cdot z}{|z|}$ is the length of
the projection onto the radial direction at $z$ 
and $h_{\Sp^{n-1}}=h-\frac{h\cdot z}{|z|^2} z$ is the component tangential to 
the sphere $\{x\in\R^n:|x|=r\}$ at $z$.
The details  become somewhat technical.
With unweighted $\R^n$ this gives an Ahlfors $n$-regular example
with nonunique Green functions for $1<p<n$.
We leave the details to the interested reader.
\end{remark}

Example~\ref{ex-Finsler-1} also proves Corollary~\ref{cor-nonls}:

\begin{proof}[Proof of Corollary~\ref{cor-nonls}]
Consider $(X,d,\mu)$ as in Example~\ref{ex-Finsler-1}.
Let $f_1(\theta) \equiv 0$ and $f_2$ be a nonnegative $1$-Lipschitz function
on $\Spe$ such that $f_2(\theta)=0$ if and only if $\theta=0$.
Next let
\[
u_j(re^{i\theta})= (r^{-a_p}-1)e^{a_p f_j(\theta)}
\quad \text{for } 0 <r < 1 \text{ and } \theta \in \Spe, \ j=1,2,
\]
where $a_p$ is as in~\eqref{eq-def-ap}.
By Example~\ref{ex-Finsler-1}, 
$u_1$ and $u_2$ are \p-harmonic in the punctured disc
$\Om:=\{x\in\R^2:0<|x|<1\}$.
Moreover, $u_1 \le u_2$ in $\Om$ and $u_1(x)=u_2(x)$ if and only if 
$x=re^{i\theta}\in\Om$ and $\theta=0$.
This  shows that the strong comparison principle fails for \p-harmonic functions
on $X$.

For the second part consider $p=2$ above and let $v=u_2 - u_1$.
Then $v \ge 0$ in $\Om$ and $v(x)=0$ if and only if 
$x=re^{i\theta}\in\Om$ and $\theta=0$.
It thus follows from the strong maximum principle that $v$ cannot
be \p-harmonic in $\Om$.
\end{proof}

\end{document}